\documentclass[11pt]{amsart}

\usepackage{amsmath,graphicx,amssymb}
\usepackage[all,cmtip]{xy}
\xyoption{all}
\usepackage[small]{caption}

\newtheorem{thm}{Theorem}[section]
\newtheorem{lem}[thm]{Lemma}
\newtheorem{cor}[thm]{Corollary}
\newtheorem{prop}[thm]{Proposition}
\newtheorem{rem}[thm]{Remark}
\theoremstyle{definition}
\newtheorem{defn}[thm]{Definition}

\DeclareMathOperator{\aut}{Aut}
\DeclareMathOperator{\sides}{{Sides}}
\DeclareMathOperator{\missing}{{Missing}}
\DeclareMathOperator{\area}{{Area}}

\DeclareMathOperator{\link}{Link}
\DeclareMathOperator{\perimeter}{\text{\textsc{Per}}}

\DeclareMathOperator{\stabilizer}{\ensuremath{Stab}}
\DeclareMathOperator{\fixer}{\ensuremath{Fix}}
\DeclareMathOperator{\added}{\ensuremath{Added}}
\newcommand{\size}[1]{\ensuremath{\vert #1 \vert}}
\newcommand{\nclose}[1]{\ensuremath{\langle\!\langle#1\rangle\!\rangle}}
\newcommand{\mc}{\mathcal }
\newcommand{\mb}{\mathbb }
\newcommand{\abrm}{(\mc A \ast \mc B)/ \nclose{r^m}}
\newcommand{\Z}{\mathbb{Z}}

\begin{document}
\title[Local Quasiconvexity of Groups]{Local Quasiconvexity of Groups acting on Small Cancellation Complexes}
\author[E.Mart\'inez-Pedroza]{Eduardo Mart\'inez-Pedroza}
      \address{%Dept. of Math. and Stat.\\
               McMaster University\\
               Hamilton, Ontario, Canada L8P 3E9}
      \email{emartinez@math.mcmaster.ca}
\author[D.~T.~Wise]{Daniel T. Wise}
      \address{%Dept. of Math. and Stat.\\
                     McGill University \\
               Montreal, Quebec, Canada H3A 2K6 }
      \email{wise@math.mcgill.ca}
\subjclass[2000]{}
\date{May 3, 2011}

\begin{abstract}
Given a group acting cellularly and cocompactly on a simply-connected  2-complex, we provide a criterion establishing that all finitely generated subgroups have quasiconvex orbits. This work generalizes the ``perimeter method''. As an application, we show that high-powered one-relator products $\mc A \ast \mc  B /  \nclose{r^n}$ are coherent if $\mc A$ and $\mc B$ are coherent.
\end{abstract}

\maketitle

\section{Introduction}
A group $\mc G$ is \emph{coherent} if each finitely generated subgroup $\mc H$ of $\mc G$ is finitely presented.
A group $\mc G$ is \emph{locally quasiconvex} if each finitely generated subgroup is quasiconvex.
Recall that a subgroup $\mc H$ of  $\mc G$ is \emph{quasiconvex}
if there is a constant $L$, such that every geodesic in the Cayley graph of $\mc G$  that joins two elements of $\mc H$ lies in an $L$-neighborhood of $\mc H$. While $L$ depends upon the choice of Cayley graph,
it is well-known that the quasiconvexity of $\mc H$ is independent of the finite generating set when $\mc G$ is hyperbolic.

 A simple method for proving coherence and local quasiconvexity was given in \cite{McWi-coherence}
  which introduced the \emph{perimeter} of a combinatorial map. One of the main applications there was the following
   \cite{McWi-coherence} (see also \cite{HruskaWise-Torsion}).
\begin{thm}\label{thm:MWcoherence}
Let $r$ be a cyclically reduced word and let $\mc G=\langle a,b, \dots | r^m \rangle$.
\begin{enumerate}
\item \label{MW:co}If $m \geq |r|-1$ then $\mc G$ is coherent.
\item \label{MW:lqc} If $m\geq 3|r|$ then $\mc G$ is locally quasiconvex.
\end{enumerate}
\end{thm}

The initial motivation in \cite{McWi-coherence} was to examine the coherence of one-relator groups with torsion, engaging with the well-known problem of whether every one-relator group is coherent. The method, however turned out to be widely applicable for
suitably deficient small-cancellation groups.

In this paper we revisit the perimeter method, and redefine it for $\mc H$-equivariant embeddings $Y\subset X$ (to the universal cover)
instead of maps $Y\rightarrow X$ (to the base space). This new approach is flexible enough to deal with torsion. In contrast,  the method in \cite{McWi-coherence} was restricted to torsion arising from defining relators that are high-powers of words.

Recall that the $C'(\lambda)$ condition on a 2-complex asserts that $|P|<\lambda|\partial R|$ whenever $P$ is a ``piece'' occurring on the boundary cycle of a 2-cell $R$. To be ``uniformly circumscribed'' means that there is an uniform upper bound on each $|\partial R|$, and  ``$M$-thin'' means that each 1-cell of $X$ lies  on the boundary of at most $M$ 2-cells. A connected subcomplex $Y$ of a 2-complex $X$ is quasi-isometrically embedded if the inclusion of 1-skeletons is a quasi-isometric embedding with respect to the combinatorial path metrics.  Precise definitions are given in Sections~\ref{sec:small-cancellation} and~\ref{sec:lq-criteria}. Our main result is then the following.
\begin{thm}[Locally Quasiconvex]\label{thm:C} %%ThmA
Let $X$ be a $C'(\lambda)$ 2-complex that is simply-connected, uniformly circumscribed, and $M$-thin.
Suppose that $6\lambda M < 1$. 

If $\mc H\subset \aut (X)$ is finitely generated [relative to a finite collection of $0$-cell stabilizers], 
then there is a quasi-isometrically embedded subcomplex $Y \hookrightarrow X$  on which $\mc H$ acts cocompactly.
\end{thm}

In comparison, the analogous result in \cite{McWi-coherence} is as follows:
\begin{thm}\label{thm:Cbase} %%ThmA
Let $\mc G = \pi_1X$ where $X$ is a $C'(\lambda)$ 2-complex that is compact and $M$-thin.
Suppose that $3\lambda M < 1$.
Then $\mc G$ is locally quasiconvex.
\end{thm}
The statements of Theorems~\ref{thm:C}~and~\ref{thm:Cbase} are almost identical except that $3\neq 6$ - a difference that disappears if we assume that $\mc H$ acts without inversions on the 1-skeleton (see Remark~\ref{rem:no-inversions}).

To get a feel for Theorem~\ref{thm:C}, let us first describe some special cases under the ordinary finite generation hypothesis.
Theorem~\ref{thm:C} applies in a variety of new situations were one would hope to apply the method of \cite{McWi-coherence}.
For instance, it implies the local quasiconvexity of groups acting properly and cocompactly on sufficiently thin 2-dimensional hyperbolic buildings and polygons of finite groups - something unobtainable directly using \cite{McWi-coherence} without knowing virtual torsion-freeness (which is not obvious \cite{WisePolygons}). Likewise it often applies when $X$ lies in a rich class of beautiful 2-complexes
 studied by Haglund in \cite{Haglund1991}: A \emph{$(p,r)$ Gromov polyhedron} $X$ is a simply-connected 2-complex
where each 2-cell is a $p$-gon, and each $0$-cell $x$ has $\link(x)\cong K(r)$. So  $X$ is $C'(\frac1{p-\epsilon})$ and $p$-circumscribed, and $(r-1)$-thin. Haglund constructed many groups acting properly and cocompactly on these Gromov polyhedra, but very few  of these are known to be virtually torsion-free. We then have:
\begin{cor}
Let $\mc G$ act properly and cocompactly on a $(p,r)$ Gromov polyhedron.
Then $\mc G$ is locally quasiconvex provided that $\frac{6}{p}(r-1)<1$.
\end{cor}

Let us now turn to the more general formulation within Theorem~\ref{thm:local-quasiconvexity} that is aimed to determine relative quasiconvexity in relatively hyperbolic groups. 

A graph $\Gamma$ is \emph{fine} if each edge of $\Gamma$ is contained in only finitely many circuits of length~$n$ for each $n$.
A countable group $\mc G$ is hyperbolic relative to a collection of subgroups $\mb P$ if there is fine and connected hyperbolic graph $\Gamma$  on which $\mc G$ acts cocompactly and with finite edge stabilizers, and  $\mathbb P$ is a set of representatives of vertex stabilizers such that each infinite vertex stabilizer is represented, we refer the interested reader to~\cite{BO99}. In \cite{MaWi10b} we showed:% the following characterization of relative quasiconvexity in relatively hyperbolic groups.
\begin{thm}[Relative Quasiconvexity Criterion]\label{def:ours} 
A subgroup $\mc H$ of $\mc G$ is \emph{quasiconvex relative to  $\mb P$} if and only if there is a non-empty connected and quasi-isometrically embedded subgraph of $\Gamma$ on which $\mc H$ acts cocompactly.
\end{thm}
In the relatively hyperbolic setting, Theorem~\ref{def:ours} allows  to interpret the locally relatively quasiconvex conclusion of Theorem~\ref{thm:C}. 
This interpretation yields the conclusion of actual local quasiconvexity or coherence provided that the parabolic subgroups have these properties. This employs the following result discussed in~\cite{Ma08}:
\begin{thm}\label{thm:M}
Let $\mc G$ be  finitely generated and hyperbolic relative to a finite collection of subgroups $\mb P$.  If every relatively finitely generated  subgroup of $\mc G$
is relatively quasiconvex, then the following statements hold:
\begin{enumerate}
\item If each $\mc P \in \mb{P}$ is  coherent, then $\mc G$ is coherent.
\item If each $\mc P \in \mb{P}$ is hyperbolic and  locally quasiconvex, then $\mc G$ is hyperbolic and locally quasiconvex.
\end{enumerate}
\end{thm}
The following application to one-relator products is proven in Section~\ref{sub:lq-products}:
\begin{thm}\label{thm:ABrelQuasiconvex}
Let $\mc A$ and $\mc B$ be countable groups,  let $r \in \mc A \ast \mc B$ be a cyclically reduced word of length at least $2$, and $m>0$ such that $3|r|<m$.

If $\mc H$ is a subgroup of $\abrm$ that is finitely generated relative to $\{\mc A, \mc B\}$, then $\mc H$ is quasiconvex relative to $\{\mc A, \mc B\}$.
\end{thm}

We now describe the application that motivated this work,
which is to generalize Theorem~\ref{thm:MWcoherence} to the context of ``one-relator products''.
%For $r^m \in \mc A\ast \mc B$, we let $\nclose{r^m}$ denote the normal closure of $r^m$ in the free product $\mc A\ast\mc B$.
The following application closely parallels Theorem~\ref{thm:MWcoherence}:
\begin{thm} \label{thm:MWparallel}
Let $\mc A$ and $\mc B$ be countable groups,  let $r \in \mc A \ast \mc B$ be a cyclically reduced word of length at least $2$, and $m>0$ such that $3|r|<m$.
\begin{enumerate}
\item If $\mc A$ and $\mc B$ are coherent, then $\abrm$ is coherent.
\item If $\mc A$ and $\mc B$ are hyperbolic and locally quasiconvex, then $\abrm$ is hyperbolic and locally quasiconvex.
\end{enumerate}
\end{thm}
\begin{proof}
The group $\abrm$ is hyperbolic relative to $\{\mc A,\mc B\}$, see for example \ref{thm:filling}.(\ref{filling-2}). Consequently, Theorem~\ref{thm:MWparallel} follows by combining Theorem~\ref{thm:ABrelQuasiconvex}  with Theorem~\ref{thm:M}.
\end{proof}

\subsection*{Acknowledgment:}
The first author acknowledges the support of the Geometry and Topology Group at McMaster University through a Postdoctoral Fellowship, and partial support of the CRM in Montreal to attend its special Fall-2010 semester during which part of this paper was prepared.
The second author's research was supported by NSERC.

\section{Disc diagram and small cancellation background}\label{sec:small-cancellation}
This paper follows the notation used in \cite{McWi-fans, McWi-coherence}, and in this section we quote various of those relevant notations.

\begin{defn}[Complexes and Automorphisms]
All complexes considered in this paper are combinatorial $2$-dimensional complexes,
and all maps are combinatorial.
If $X$ is a $2$-complex then $\aut (X)$ denotes the group of cellular automorphisms of $X$.
\end{defn}

\begin{defn}[Path and cycle]\cite[Def 2.6]{McWi-fans}
A \emph{path} is a map $P\rightarrow X$ where $P$ is a subdivided
interval or a single $0$-cell.  In the latter case, $P$ is
\emph{trivial}.  A \emph{cycle} is a map $C\rightarrow X$
where $C$ is a subdivided circle.  Given two paths $P\rightarrow
X$ and $Q\rightarrow X$ such that the terminal point of $P$ and
the initial point of $Q$ map to the same $0$-cell of $X$, their
concatenation $PQ\rightarrow X$ is the obvious path whose domain
is the union of $P$ and $Q$ along these points.  The path
$P\rightarrow X$ is \emph{closed} if the
endpoints of $P$ map to the same $0$-cell of $X$.  A path or cycle
is \emph{simple} if the map is injective on $0$-cells.
The \emph{length} of the path $P$ or cycle $C$ is the number of
$1$-cells in the domain and is denoted by $\size{P}$ or
$\size{C}$.  The \emph{interior} of a path is the path minus its
endpoints.  In particular, the $0$-cells in the interior of a path
are the $0$-cells other than the endpoints.  A \emph{subpath} $Q$
of a path $P$ [or a cycle $C$] is given by a path $Q \rightarrow P
\rightarrow X$ [$Q \rightarrow C \rightarrow X$] in which distinct
$1$-cells of $Q$ are sent to distinct $1$-cells of $P$ [$C$].
Note that the length of a subpath is at most that of the path
[cycle] containing it.  A nontrivial
closed path determines a cycle in the obvious way.  Finally, when
the target space is understood we will usually refer to
$P\rightarrow X$ as the path $P$.
\end{defn}

\begin{defn}[Disc Diagram]\cite[Def 7.4]{McWi-coherence}\label{def:diagrams}
A {\em disc diagram} $D$ is a compact contractible $2$-complex with a
fixed embedding in the plane.  A {\em boundary cycle} $P$ of $D$ is a
closed path in $\partial D$ which travels entirely around $D$ (in a
manner respecting the planar embedding of $D$).

Let $P\rightarrow X$ be a closed null-homotopic path.  A {\em disc
diagram in $X$ for $P$} is a disc diagram $D$ together with a map
$D\rightarrow X$ such that the closed path $P\rightarrow X$ factors as
$P\rightarrow D\rightarrow X$ where $P\rightarrow D$ is the boundary
cycle of $D$.  The van Kampen lemma \cite{Ka33} essentially states
that every null-homotopic path $P\rightarrow X$ is the boundary cycle
of a disc diagram.  Define $\area(D)$ to be the number of $2$-cells
in $D$.  For a null-homotopic path $P\rightarrow X$, we define
$\area(P)$ to equal the minimal number of $2$-cells in a disc diagram
$D\rightarrow X$ that has boundary cycle $P$.  The disc diagram
$D\rightarrow X$ is then a {\em minimal area disc
diagram} for $P$.
\end{defn}

\begin{defn}[Piece]\cite[Def 3.1]{McWi-fans}\label{def:piece}
Let $X$ be a combinatorial $2$-complex.  Intuitively, a piece of $X$
is a path which is contained in the boundaries of the $2$-cells of $X$
in at least two distinct ways.  More precisely, a nontrivial path
$P\rightarrow X$ is a \emph{piece} of $X$ if there are $2$-cells $R_1$
and $R_2$ such that $P\rightarrow X$ factors as $P \rightarrow R_1
\rightarrow X$ and as $P\rightarrow R_2\rightarrow X$ but there does
not exist a homeomorphism $\partial R_1\rightarrow \partial R_2$ such
that there is a commutative diagram:
\begin{equation*}\label{eq:piece}
\begin{array}{ccc}
P             & \rightarrow & \partial R_2\\
\downarrow    & \nearrow    & \downarrow\\
\partial R_1 & \rightarrow & X
\end{array}
\end{equation*}
\end{defn}

\begin{defn}[$C'(\lambda)$ complex]
For a fixed positive real number $\lambda$, the complex $X$ satisfies
$C'(\lambda)$ provided that for each $2$-cell
$R\rightarrow X$, and each piece $P\rightarrow X$ that factors as
$P\rightarrow R\rightarrow X$, we have $\size{P}<\lambda \size{\partial
R}$.
\end{defn}

\begin{figure}\centering
\includegraphics[width=.6\textwidth]{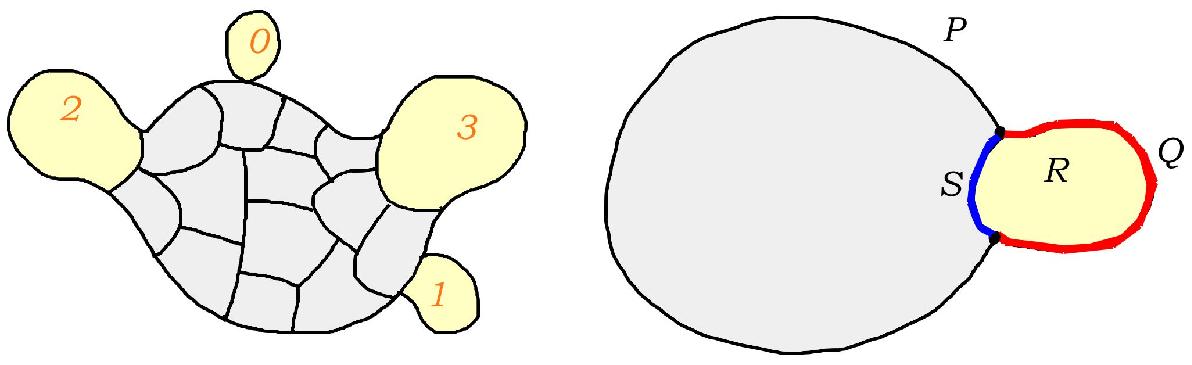}
\caption{Various $i$-shells are indicated on the left. $S$ and $R$ denote the inner and outer paths of the $i$-shell $R$.\label{fig:Shells}}
\end{figure}

\begin{defn}[$i$-shells and spurs]\cite[Def 9.3]{McWi-coherence}\label{def:i-shells}
An \emph{$i$-shell} in a disc diagram $D$ is a $2$-cell $R \hookrightarrow D$ whose boundary cycle $\partial R$ is the concatenation $Q S_1\cdots S_i$ where $Q  \rightarrow D$ is a boundary arc,
the interior of $S_1\cdots S_i$ maps to the interior of $D$,
and  $S_j \rightarrow D$ is a nontrivial interior arc of $D$ for all $j > 0$.
The path $Q $ is the \emph{outer path} of the $i$-shell. See Figure~\ref{fig:Shells}.

A $1$-cell $e$ in $\partial D$ that is incident with a valence~$1$ $0$-cell $v$ is a \emph{spur}.
\end{defn}

\begin{defn}[Arc]\cite[Def 5.4]{McWi-fans}\label{def:sc}
An \emph{arc} in a diagram $D$ is a path $P\rightarrow D$  such that each of its interior 0-cells 
is mapped to a 0-cell with valence 2 in $D$.  An arc which is not a proper subpath of any other arc 
is a \emph{maximal arc}. The arc is \emph{internal} if its interior lies in the interior of $D$,
and it is a \emph{boundary arc} if it lies entirely in $\partial D$.
\end{defn}

\begin{defn}[Doubly-based diagram, Cut tree, Ladder] \cite[Def 5.1, 5.3]{McWi-fans}%\label{def:ladder}
A \emph{doubly-based diagram $D$} is a disc diagram in which two (possibly identical) $0$-cells, $s$ and $t$, have been
specified in the boundary cycle of $D$. The $0$-cells $s$ and $t$ are called the \emph{basepoints} of $D$.
The paths $P_1 \rightarrow D$ and $P_2 \rightarrow D$ with $s$ as their common startpoint
and $t$ as their common endpoint and such that $P_1P_2^{-1}$ is the boundary cycle of $D$ are the
\emph{boundary paths determined by the basepoints of $D$.}

The \emph{cut-tree $T$} of a disc diagram $D$ is defined as follows. A $0$-cell $v$ is called a
cut $0$-cell of $D$ provided that $D -\{ v\}$ is not connected.
 Let $V$ be the set of all cut $0$-cells of $D$. A connected component of
$D -\{V\}$ is a \emph{cut-component}. Let $C$ be the set of cut-components of
$D$. The tree $T$ is constructed by adding a black $0$-cell for each $0$-cell $v \in V$
and a red $0$-cell for each component $c \in C$. A $1$-cell connects the $0$-cell for
$v$ to the $0$-cell for $c$ if and only if $v$ is in the closure of $c$. Since each  black $0$-cell disconnects $T$, the graph is a tree.
\end{defn}

\begin{defn}[Ladder] \cite[6.1]{McWi-fans}\label{def:ladder}
Let $D$ be a doubly-based diagram. Suppose that $D$ is not a single 2-cell, suppose that the basepoints of $D$ are distinct and are not cut 0-cells,
and suppose that its cut tree is either trivial or a subdivided interval. Suppose further that if the cut tree is a subdivided interval then the basepoints lie
in the cut components corresponding to the endpoints of the interval. 

Let  $P_1\rightarrow D$ and $P_2\rightarrow D$ be the two boundary paths determined by the basepoints of $D$. 
Then $D$ is called a \emph{ladder} if every maximal internal arc of $D$ begins at a $0$-cell in the interior of $P_1 \rightarrow D$ and ends at a $0$-cell
in the interior of $P_2 \rightarrow D$.  
\end{defn}

\subsection{Greendlinger's lemma}
The following classification of disc diagrams summarizes the basic tool in $C'(\frac16)$ small cancellation theory:
\begin{thm}\cite[Thm 9.4]{McWi-coherence, McWi-fans} \label{thm:fan-classification}
If $D$ is a $C'(1/6)$ disc diagram, then one of the following holds:
\begin{enumerate}
\item $D$ contains at least three spurs and/or $i$-shells with $i\leq 3$.
\item $D$ is a ladder, and hence has a spur, $0$-shell or
$1$-shell at each end.
\item $D$ consists of a single $0$-cell or a single $2$-cell.
\end{enumerate}
\end{thm}

The following well-known consequence of Theorem~\ref{thm:fan-classification} is easily verified:
\begin{cor}\label{cor:boundary cycles embed} Let $X$ be a simply-connected $C'(1/6)$ complex.
Then the boundary cycle of each $2$-cell embeds in $X$.
\end{cor}

\subsection{Missing shells, and Quasi-isometric Embedding Criterion}

\begin{figure}\centering
\includegraphics[width=.5\textwidth]{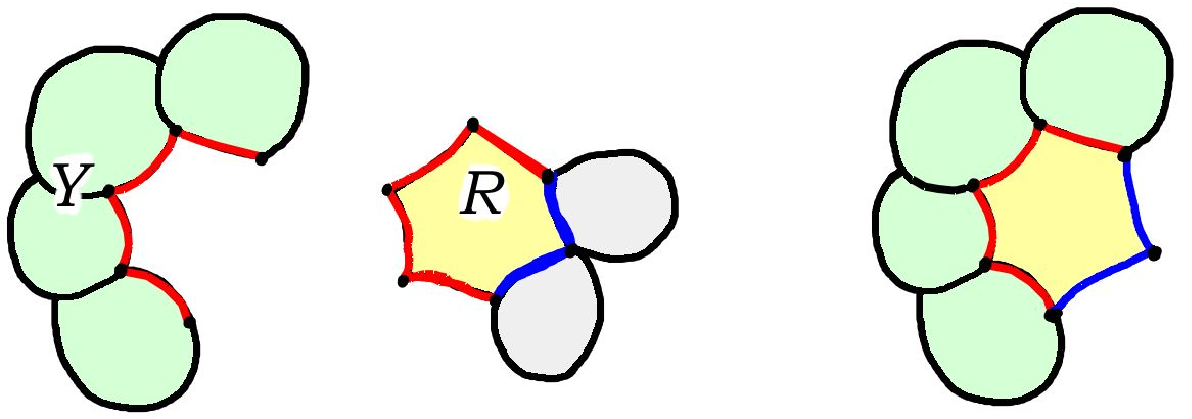}
\caption{A missing $2$-cell $R$ is attached to the subspace $Y$.\label{fig:MissingShell}}
\end{figure}

\begin{defn}[Missing $i$-shell]
Let $X$ be a 2-complex, and $Y$ a subcomplex of $X$. A $2$-cell $R$ of $X$ is a \emph{missing $i$-shell of $Y$} if
$\partial R=QS$ where $Q$ is a path in $Y$, $S$ is the concatenation of at most $i$-pieces of $X$, and $R$ is not contained in $Y$.
The paths $Q$ and $S$ are  the \emph{outer path} and \emph{inner path} of the missing shell $R$ respectively. See Figure~\ref{fig:MissingShell}.
\end{defn}

\begin{defn}[Quasi-isometric Embedding]
Let $X$ be a connected 2-complex. A \emph{geodesic} between the 0-cells $u,v$ is a path $P\rightarrow X$ 
of minimal length among all possible paths between $u$ and $v$.

Let $Y$ be a connected subcomplex of $X$, and let $L$ be a positive constant. The inclusion $Y \rightarrow X$ is an \emph{$L$-quasi-isometric embedding}
if for any pair of 0-cells $u,v$ of $Y$ and any pair of geodesics  $P_1\rightarrow Y$  and  $P_2\rightarrow X$ between $u$ and $v$, we have
$|P_1| \leq L|P_2|$.
\end{defn}

\begin{lem}\label{lem:ladder-area}
Let $D$ be a ladder with no 2-shells, let  $P_1$ and $P_2$ be the boundary paths of $D$,  and let $L>0$ be an integer such that $|\partial R|<L$ for each $2$-cell $R \subset D$.  Then $\area (D) \leq |P_2|$ and $|P_1| \leq L |P_2|$.
\end{lem}
\begin{proof}
Since $D$ is a ladder, different pieces have disjoint interiors, and the boundary of each 2-cell of $D$ contains at most two pieces. 
Moreover,  if the boundary of a 2-cell $R$ contains at most one piece, then $\partial R$  intersects both boundary paths in non-trivial subpaths.
If the boundary of a 2-cell $R$ contains two pieces, then either $R$ is a 2-shell or $\partial R$ intersects both boundary paths in non-trivial subpaths.

Since $D$ has no 2-shells, the boundary of each 2-cell $R$ of $D$ intersects  both $P_1$ and $P_2$ in non-trivial subpaths, 
and therefore $\area (D) \leq |P_i|$ for $i=1,2$.  For the second inequality, 
observe that each 1-cell of $P_1$ either belongs to $P_2$ or is contained in the boundary of a 2-cell of $D$, therefore
 \[ |P_1| \leq |P_2| + (L-1) \area (D) \leq L|P_2|.  \qedhere  \]
\end{proof}

The following is a variation of the quasiconvexity criterion in \cite{McWi-coherence}:

\begin{prop}[Quasi-isometric Embeddedness Criterion] \label{lem:No-shells-convexity}
Let $X$ be a $C'(1/6)$ 2-complex that is simply-connected, and suppose that there $L>0$ such that $|\partial R|<L$ for each $2$-cell $R \subset X$.
Let  $Y$ be a connected subcomplex of $X$ with no missing $3$-shells.
Then the inclusion $Y \rightarrow X$ is a $L$-quasi-isometric embedding.
\end{prop}
\begin{proof}
Let $P_1\rightarrow Y$  and  $P_2\rightarrow X$ be geodesics with the same endpoints. Let $D\rightarrow X$ be a reduced disc diagram with boundary cycle $P_1P_2^{-1}$, see Figure~\ref{fig:LadderQuasiconvexity}. We point out three observations and then we conclude:

\emph{If $R \hookrightarrow D$ is an $i$-shell of $D$ with $i\leq 3$, then the outer path of $R$ intersects $P_1$ and $P_2$ in non-trivial subpaths. }  
Indeed, the outer path of $R$  cannot be a subpath of $P_2$ since this would contradict that $P_2$ is a geodesic in $X$ - since the inner path of an $i$-shell is shorter than the outer path when $i\leq 3$.  Similarly, that $Y$ has no missing $i$-shells for $i\leq 3$ implies that the outer path of $R$  cannot be a subpath of $P_1$, since otherwise $R$ would also lie in $Y$ thus ensuring that $P_1$ can be shortened in $Y$.  

\emph{If a 1-cell $e$ is a spur in $\partial D$, then $e$ is a common 1-cell of $P_1$ and $P_2$ located either at the start point or end point of the paths $P_1,P_2$.}
Since $P_1$ and $P_2$ are geodesics with the same endpoints, these paths do not backtrack 1-cells. Therefore spurs can be only at the  common start point of $P_1$ and $P_2$, or at the common end point of $P_1$ and $P_2$.

\emph{The diagram $D$ is a  single 0-cell, a single 2-cell,  or a ladder with no 2-shells.}
Indeed, since $D$ is a $C'(1/6)$ disc diagram with at most two spurs and/or $i$-shells with $i\leq 3$, this follows directly from Theorem~\ref{thm:fan-classification}.

\emph{Conclusion.}
If $D$ is a single $2$-cell or a single $0$-cell, then clearly $|P_1|\leq L|P_2|$. Otherwise, $D$ is a ladder with no 2-shells, and then Lemma~\ref{lem:ladder-area} implies that $|P_1| \leq L|P_2|$.
\end{proof}

\begin{figure}\centering
\includegraphics[width=.45\textwidth]{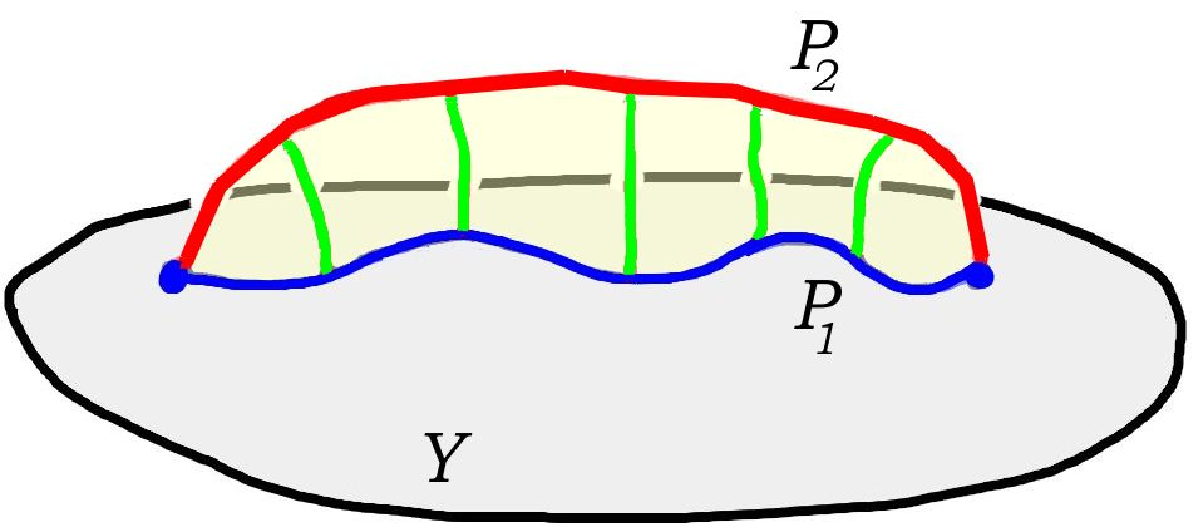}
\caption{A minimal area disc diagram between a geodesic $P_2$ and a space $Y$ with no missing shells,  is a ladder.\label{fig:LadderQuasiconvexity}}
\end{figure}

\section{Quasiconvexity and Perimeter-reduction}\label{sec:lq-criteria}

Given a simply-connected combinatorial 2-complex $X$, we provide a criterion for verifying that all (relatively) finitely generated subgroups of $\aut (X)$ have quasiconvex orbits. If the group acts freely and cocompactly, this coincides with a criterion from~\cite{McWi-coherence} to determine local quasiconvexity
for small cancellation groups; our approach extends some of those techniques.

\begin{defn}[Circumscribed]
A $2$-complex $X$ is \emph{$L$-circumscribed} if there exists an integer $L$ such that
for each $2$-cell $R$ of $X$, the boundary cycle $\partial R$ has length at most $L$.
We say that $X$ is \emph{uniformly circumscribed} if $X$ is $L$-circumscribed for some $L$.
\end{defn}

\begin{defn}[Thinness]
A $2$-complex $X$  is \emph{thin} if $\sides_X (x)$ is a finite set for every $1$-cell $x$ in $X$.
If there exists an integer $M$ such that $|\sides_X (x)|\leq M$ for every $1$-cell $x$ in $X$, then we say that $X$
is \emph{$M$-thin}.
All $2$-complexes considered in this paper are thin and most are $M$-thin for some $M$.
\end{defn}

\begin{thm}[Local Quasiconvexity]\label{thm:local-quasiconvexity}
Let $X$ be a $C'(\lambda)$ complex that is simply-connected, uniformly circumscribed, and $M$-thin.
Suppose that $6\lambda M < 1$.

If $\mc H<\aut (X)$ is finitely generated relative to a finite collection of $0$-cell stabilizers. 
Then there exists a connected and quasi-isometrically embedded $\mc H$-cocompact subcomplex of $X$.
\end{thm}

\begin{rem}
When $\aut (X)$ acts without inversions on  $X^1$, then Theorem~\ref{thm:local-quasiconvexity} holds under the weaker hypothesis $3\lambda M<1$.
See Remark~\ref{rem:no-inversions}.

We could develop parallel $C(4)$-$T(4)$ results where  $\leq2$-shells play the role of $\leq3$-shells etc.
And there are conditions that ensure perimeter reductions.
This was discussed in detail in \cite{McWi-coherence}.
\end{rem}

\begin{proof}[Proof of Theorem~\ref{thm:local-quasiconvexity} and description of the rest of the section.]
That $X$ is an $M$-thin simply-connected $C'(\lambda)$-complex with $6\lambda M <1$ imply that $X$ satisfies what we called the \emph{Perimeter-reduction hypothesis}. This hypothesis and the stated result are the main contents of Section~\ref{subsec:reduction-criterion}.

Then the main result of Section~\ref{subsection:lqc-criterion} states that  any $L$-circumscribed, thin and simply-connected $C'(\lambda)$-complex satisfying the perimeter-reduction hypothesis, satisfies the conclusion of Theorem~\ref{thm:local-quasiconvexity}.
\end{proof}

\subsection{Perimeter with respect to a group action}\label{subsec:perimeter}

The following Definition modifies the notation introduced in~\cite[Conv~2.7, Def~2.8, and Rem~2.9]{McWi-coherence}.

\begin{defn}[Sides] \cite[Def 2.8]{McWi-coherence} \label{def:side}   \label{defn:side-subcomplex}
Let $X$ be a $2$-complex, and let $R$ be a $2$-cell of $X$.  Let
$r$ be a $1$-cell in $\partial R$ and let $x$ be the image of $r$ in
$X$.  The pair $(R,r)$ is a \emph{side of a $2$-cell of
$X$ that is present at $x$}.  The collection of all sides of $X$
that are present at $x$ will be denoted by $\sides_X (x)$, and the
full collection of sides of $2$-cells of $X$ that are present at
$1$-cells of $X$ will be denoted by $\sides_X$.

Suppose that $Y$ is a subcomplex of $X$ and $(R,r)$ is a side of $X$
present at the $1$-cell $x$ of $X$. If $x$ is contained in $Y$ and
the map $(R,r) \rightarrow (X,x)$ factors through the inclusion $(Y,x) \hookrightarrow (X,x)$
then  \emph{the side $(R,r) \rightarrow (X,x)$ lifts to $Y$.}
The collection of all sides of $X$ that are present at $x$ and lift to $Y$
is denoted by $\sides_X (Y,x)$. The collection of sides of $X$ that are present at
$x$ and do not lift to $Y$ is denoted by $\missing_X (Y,x)$.

Notice that  if $x$ does not lift to $Y$ then $\sides_X (Y,x)$ is the empty set.
\end{defn}

\begin{defn}[$\mc H$-Cocompact subcomplex]
Let $X$ be a 2-complex, and $\mc H$ a subgroup of $\aut (X)$.
A subcomplex $Y \hookrightarrow X$ is  \emph{$\mc H$-cocompact } if
$Y$ is $\mc H$-invariant and  $\mc H$ acts cocompactly on $Y$.
\end{defn}

\begin{figure}\centering
\includegraphics[width=.6\textwidth]{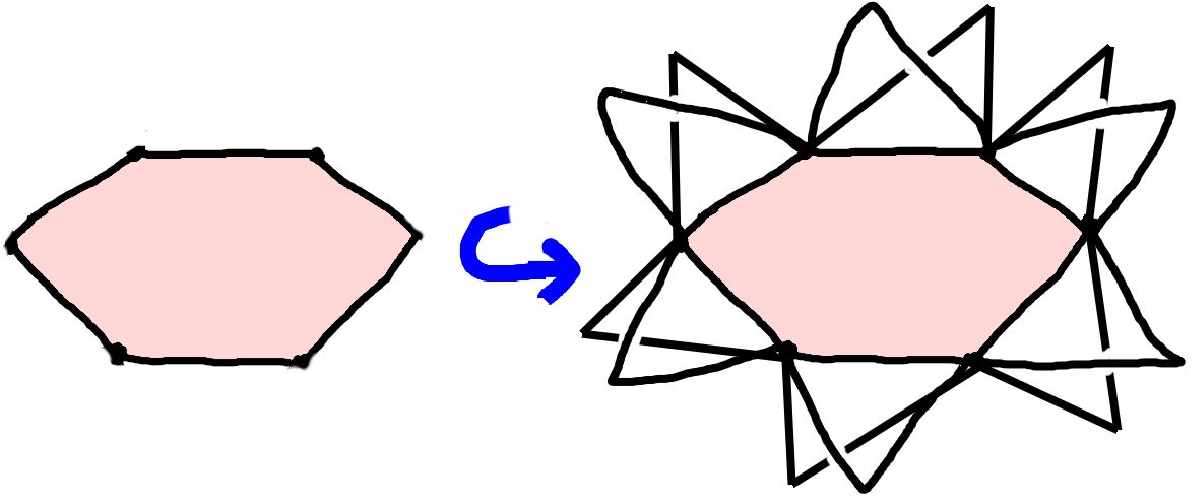}
\caption{ $\perimeter (Y, 1) = 12$, for $Y \hookrightarrow X$ and the trivial group.
Various actions of $\Z_2$ yield perimeters $6$, $8$, and $12$. \label{fig:ActionStar}}
\end{figure}

\begin{defn}[Perimeter of $\mc H$-Cocompact subcomplexes]\label{def:Perimeter}
Let $X$ be a thin 2-complex. Let  $\mc H$ be a subgroup of $\aut (X)$, and let
$Y$ be a $\mc H$-cocompact subcomplex of $X$.

By cocompactness, the action of $\mc H$ on $Y$ has finitely many $1$-cell orbits. Suppose there are $n$ orbits and let
$y_1, \dots , y_n$ be $1$-cells of $Y$ representing these orbits. Define the \emph{perimeter of $Y$ with respect to $\mc H$} to be:
\begin{equation}\label{eq:perimeter definition}
  \perimeter (Y, \mc H) = \sum_{i=1}^n |\missing_X (Y, y_i)|
\end{equation}
See Figure~\ref{fig:ActionStar}.
 \end{defn}
We note that Definition~\ref{def:Perimeter} is a slight modification of \cite[Def~2.10]{McWi-coherence} that allows us to deal with subcomplexes admitting cocompact actions.

\begin{lem}\label{lem:finite-perimeter}
Let $X$ be a thin 2-complex. Let $\mc H$ be a subgroup of $\aut (X)$, and let
$Y$ be a $\mc H$-cocompact subcomplex of $X$. Then the perimeter $\perimeter (Y , \mc H)$ is a well-defined non-negative integer.
\end{lem}
\begin{proof}
Since $X$ is thin, the sum in Equation~\eqref{eq:perimeter definition} involves only non-negative integers and hence $\perimeter (Y , \mc H)$ is a non-negative integer. The sum is well-defined because
there is a bijection $\missing_X (Y, y) \leftrightarrow \missing_X (Y, h.y)$
for each $h\in \mc H$ and $1$-cell $y$ of $Y$.
\end{proof}

\subsection{The perimeter-reduction criterion theorem:%$(R,\mc H)$-enlargement reduces perimeter
} \label{subsec:reduction-criterion}

\begin{defn}[Perimeter-reduction hypothesis]
A thin $2$-complex $X$ satisfies the \emph{Perimeter-reduction hypothesis} if the following property holds:
For any subgroup $\mc H$ of $\aut (X)$ that is finitely generated relative to a finite collection of $0$-cell stabilizers,
and any  connected $\mc H$-cocompact subcomplex $Y \subset X$ with a missing $3$-shell,
there is a connected $\mc H$-cocompact subcomplex $Y' \subset X$ with $\perimeter (Y', H) < \perimeter (Y, H)$.
\end{defn}

\begin{thm}[Perimeter-reduction criterion]\label{thm:pr-criteria}
Let $X$ be a $C'(\lambda)$ complex that is simply-connected, $M$-thin, and satisfies that $6\lambda M < 1$.
Then $X$ satisfies the Perimeter-reduction hypothesis.
\end{thm}

Theorem~\ref{thm:pr-criteria} is an immediate consequence of Proposition~\ref{prop:perimeter} whose proof
is the goal of this section.

\begin{defn}[$\mc H$-enlargement] \label{defn:packed-filling}
Let $X$ be a 2-complex, let $Y$ be a subcomplex of $X$, let $R$ be a 2-cell of $X$, and let $\mc H < \aut (X)$.
The  \emph{$(\mc H, R)$-enlargement of $Y$} is the subcomplex $Y'$ of $X$ obtained by adding all $\mc H$-translates of $R$  as follows:
\[   Y' = Y \cup  \bigcup_{h \in H} h.R ,\]
\end{defn}
\begin{prop}[$(R,\mc H)$-enlargement reduces perimeter]\label{prop:perimeter}
Let $X$ be a $C'(\lambda)$ complex that is simply-connected, $M$-thin, and satisfies $6\lambda M < 1$.

Let $\mc H<\aut (X)$, let $Y\subset X$ be $\mc H$-cocompact, and let $R \subset X$ be a missing $3$-shell of $Y$.
Then the $(\mc H, R)$-enlargement $Y'$ of $Y$ satisfies:
\begin{equation}\label{ineq:perimeter}
 \perimeter (Y', \mc H) <  \perimeter (Y, \mc H) .
\end{equation}
\end{prop}

{\bf Plan:} The proof is divided into two cases depending upon the group $\aut_{\mc H}(R)$ defined below.
When $\aut_{\mc H}(R)$ is a large subgroup (of the dihedral group $\aut(R)$) then  Proposition~\ref{prop:perimeter} is obvious
as we show below that  no new 1-cells are added. The main part of the proof is in the case where $\aut_{\mc H}(R)$ is either trivial or generated by a reflection.
This case requires a computation showing that the perimeter decreases.  The proof of Proposition~\ref{prop:perimeter} is discussed after the following three lemmas. 

\begin{defn}
Let $\mc H<\aut (X)$, and let $R \subset X$ be a $2$-cell.  We define $\aut_{\mc H}(R)$ to be the following quotient group: \[\aut_{\mc H}(R) = \stabilizer_{\mc H}(R)/ \fixer_{\mc H}(R)\] where the first group is the usual stabilizer of $R$  in ${\mc H}$ and the second is the point-wise stabilizer of $R$ in ${\mc H}$. When all boundary cycles of $2$-cells are embedded (as is the case when $X$ is a simply-connected $C'(1/6)$-complex by Corollary~\ref{cor:boundary cycles embed}), there is a natural classification of elements of $\aut_{\mc H}(R)$ as rotations or reflections. In particular, if $\aut_{\mc H}(R)$ has no rotations then it is either trivial or is generated by a reflection.
\end{defn}

\begin{lem}[Entire Circle]\label{lem:3.14159}
Let $Y\subset X$, and let $R$ be a 2-cell of $X$ with $\partial R=QS$ and $Q \subset Y$.
If $|S|<|Q|$ and $\aut_{\mc H}(R)$ contains a nontrivial rotation, then $\partial R$ lies inside $Y$.
\end{lem}
\begin{proof}
In any circle,  the translates of an arc of length more than half the circumference by the powers of a nontrivial rotation cover the entire circle.
Therefore the translates of $Q$ by elements of $\mc H$ cover $S$. Hence,
$S \subset Y$ and, in particular, $\partial R \rightarrow \subset Y$.
\end{proof}

\begin{lem}[Counting Sides]\label{lem:counting-sides}
Let $\mc H<\aut (X)$, let $Y\subset X$ be $\mc H$-cocompact, and let $R \subset X$ be a missing $3$-shell of $Y$.
Let $e$ be a $1$-cell of $\partial R$, let $\{e_1,\dots, e_{m_e}\}$ be all the $\mc H$-translates of $e$ in $\partial R$,
and let \[ \added(e) \ = \ \sides_X (Y', e) - \sides_X (Y, e).\]
Then
\begin{equation*}\label{eq:big fuss}
 |\added (e)| \geq   \frac{m_e}{|\aut_{\mc H}(R)|}.
\end{equation*}
\end{lem}
\begin{proof}
First notice that $\aut_{\mc H}(R)$ acts on $\{e_1,\dots, e_{m_e}\}$. Define a map $\{e_1,\dots, e_{m_e}\} \rightarrow \added (e)$ as follows. For each $e_i$ choose  $g_i \in \mc H$ such that $e = g_i.e_i$, and let the side $e_i$ map to the side $(g_i.R, e)$ in $\added (e)$.  Notice that if $e_i$ and $e_j$ map to the same side in $\added(e)$ then there is an element  $h\in \mc H$ such that $h(R,e_i)=(R,e_j)$, and, in particular, $h\in \stabilizer_{\mc H}(R)$.  It follows, if $e_i$ and $e_j$ map to the same side in $\added(e)$, then $e_i,\ e_j$ are in the same $\aut_{\mc H}(R)$-orbit. Therefore the preimage of each element of $\added (e)$ has cardinality at most $|\aut_{\mc H}(R)|$.
\end{proof}
\begin{rem}
For the interested reader, an exact computation of $\added (e)$ follows from a similar argument. The precise formula is given by:
\begin{equation*}
 |\added (e)| =  m_e \frac{ \big[ \stabilizer_{\mc H}(e)  : \fixer_{\mc H}(R) \big]}{\big[ \stabilizer_{\mc H}(R)   : \fixer_{\mc H}(R) \big] }
\end{equation*}
\end{rem}

\begin{lem}\label{lem:perimeter-estimation}
Let $\mc H<\aut (X)$, let $Y\subset X$ be $\mc H$-cocompact, and let $R \subset X$ be a missing $3$-shell of $Y$ with inner path $S$
and outer path $Q$, and let $Y'$ be the $(\mc H, R)$-enlargement of $Y$.
 If $|S|<3 \lambda |R|$ and $3  \lambda M  < \frac{1}{\aut_{\mc H}(R)}$, then:
\begin{equation}\label{eq:perimeter estimation}
 \perimeter (Y', \mc H) <  \perimeter (Y, \mc H) .
 \end{equation}
\end{lem}
\begin{proof}
Let $E$ be a maximal subset of 1-cells of $\partial R$ that represent distinct $\mc H$-orbits of $1$-cells in $X$.
As all new 1-cells lie in $\mc H S$, and all new 2-cells are translates of $R$, we have the following:
\begin{equation}\label{eq:Wise}
\begin{split}
 \perimeter (Y', \mc H) &  \leq \perimeter (Y, \mc H) +  \perimeter (S, 1)  - \sum_{e \in E } |\added (e)|\\
\end{split}
\end{equation}

To verify Equation~\eqref{eq:perimeter estimation}
it therefore suffices to demonstrate Equation~\eqref{eq:Martinez}.
Note that the first inequality in Equation~\eqref{eq:Martinez} follows by combining Equations~\eqref{eq:inversion-remark}~and~\eqref{eq:added},
and the second inequality follows from our hypothesis.
\begin{equation}\label{eq:Martinez}
\perimeter (S, 1)  - \sum_{e \in E } |\added (e)| \leq |\partial R| \bigg( 3\lambda M - \frac{1}{|\aut_{\mc H}(R)|} \bigg)   < 0.
\end{equation}

\begin{equation}\label{eq:inversion-remark}
 \perimeter (S, 1) = \sum_{q \in \text{Edges}(S)} \sides(X, q) \leq M |S| \leq 3\lambda M |\partial R|
\end{equation}

\begin{equation}\label{eq:added}
\sum_{e \in E } |\added (e)|    \geq  \sum_{e \in E } m_e \frac{1}{|\aut_{\mc H}(R)|} =   \frac{|\partial R|}{|\aut_{\mc H}(R)|}.
\end{equation}
Equation~\eqref{eq:inversion-remark} holds by combining the hypotheses on thinness and length of $S$.
Equation~\eqref{eq:added} follows from Lemma~\ref{lem:counting-sides} using a partition of the 1-cells in $\partial R$.
\end{proof}

\begin{proof}[Proof of Proposition~\ref{prop:perimeter}]
It follows from the definitions that $Y'$ is a connected $\mc H$-cocompact subcomplex of $X$, and hence $\perimeter (Y', \mc H)$
is defined. Let $S$ and $Q$ be the inner and outer paths of the missing $3$-shell $R$ of $Y$.
Since $X$ is $C'(1/6)$ and $S$ is the concatenation of at most three pieces of $\partial R$,
\[ |S| < 3\lambda |\partial R| \leq \frac{|\partial R|}{2}.\]

If $\aut_{\mc H} (R)$ contains a rotation, then Lemma~\ref{lem:3.14159} implies that
$Y$ and $Y'$ have the same 1-skeleton, and therefore Equation~\ref{ineq:perimeter} follows immediately.

Suppose $\aut_{\mc H} (R)$ is trivial or generated by a reflection. The hypotheses  imply that
$3  \lambda M  < \frac{1}{\aut_{\mc H}(R)}$,  so Equation~\eqref{ineq:perimeter} follows from Lemma~\ref{lem:perimeter-estimation}.
\end{proof}

\begin{rem}\label{rem:no-inversions}
A strengthened version of Equation~\eqref{eq:Wise} is:
\begin{equation*}
 \perimeter (Y', \mc H)  \leq \perimeter (Y, \mc H) +  \perimeter (S, \aut_{\mc H}(R))  - \sum_{e \in E } |\added (e)| .
\end{equation*}
When $\aut (X)$ acts without inversions on the 1-skeleton of $X$, then Equation~\eqref{eq:inversion-remark} is strengthened to:
\begin{equation*}
 \perimeter (S, \aut_{\mc H}(R))  \leq \frac{M|S|}{|\aut_{\mc H}(R)|} \leq \frac{3\lambda M |\partial R|}{|\aut_{\mc H}(R)|}
\end{equation*}
Consequently, Lemma~\ref{lem:perimeter-estimation} holds under the weaker hypotheses:
\[ |S|<3 \lambda |R| \text{\  and  \ }  3\lambda M  < 1.\]
\end{rem}

\subsection{The Local Quasiconvexity Theorem}\label{subsection:lqc-criterion}

\begin{defn}[Relative finite generation]
Let $X$ be a 2-complex, and let $\mc H$ be a subgroup of $\aut( X)$.  We say that $\mc H$ is \emph{finitely generated relative to $0$-cell stabilizers}
if there is a finite number of $0$-cells $v_1, \dots ,v_n$ and a finite subset $S \subset \mc H$ such that
$S\cup \bigcup_{i=1}^n \mc H_{v_i}$ is a generating set for $\mc H$.
We use the notation $\mc H_{v}=\stabilizer_{\mc H}(v)$.
\end{defn}

\begin{thm}[Local Quasiconvexity Criterion]\label{thm:lqc-criteria}
Let $X$ be a $C'(1/6)$ complex that is simply-connected, thin,  $L$-circumscribed, and satisfies the Perimeter-reduction hypothesis.

If $\mc H<\aut (X)$ is finitely generated relative to a finite collection of $0$-cell stabilizers. Then there exists a connected and quasi-isometrically embedded $\mc H$-cocompact subcomplex of $X$.
\end{thm}

The proof is discussed after the following two lemmas. 
\begin{lem}[Initial subcomplex]\label{lem:Connected-Complex}
Let $X$ be a connected thin 2-complex. Let $\mc H$ be a subgroup of $\aut( X)$
and suppose that $\mc H$  is finitely generated relative to a finite collection of $0$-cell stabilizers.

If $C$ is a compact subcomplex of $X$, then there is a connected and compact subcomplex $Y_0$
containing $C$ such that:
\begin{enumerate}
\item $\mc H$ is finitely generated relative to the stabilizers of a collection of $0$-cells of $Y_0$, and
\item $Y=\bigcup_{g \in \mc H} g Y_0$ is a connected $\mc H$-cocompact subcomplex of $X$.
\end{enumerate}
\end{lem}
\begin{proof}
As $\mc H$ is finitely generated relative to $0$-cell stabilizers, there is a subset $S=\{g_1, \dots , g_m\} \subset \mc H$
and   $0$-cells $x_1, \dots , x_n$ of $X$ such that $\mc H$ is generated by $S \cup \bigcup_{i=1}^n \mc H_{x_i}$,
where $\mc H_{x}$ denotes the stabilizer of $x$ in $\mc H$.

The idea is to choose a subcomplex $Y_0$ with the property that $aY_0\cap Y_0 \neq \emptyset$ for each the generators
chosen above.
Since $X$ is connected, there is a connected compact subcomplex $Y_0$ containing $C$ and the set of vertices
\[ \{ x_0 \} \cup \{g_i.x_0 | 1\leq i\leq m \} \cup  \{g_i.v_j | 1\leq i\leq m, \  1\leq j \leq n \}. \]
Since $Y_0$ is compact, $Y=\bigcup_{g \in \mc H} g.Y_0$ is a $\mc H$-cocompact subcomplex of $X$.
It is straight forward to show that for each $a$ in the generating set we have that $Y_0 \cap gY_0 $, and therefore $Y$ is connected.
\end{proof}

\begin{lem}[Terminal Subcomplex]\label{lem:No-shells-Complex}
Let $X$ be a connected  thin 2-complex that satisfies the perimeter-reduction hypothesis.
Let $\mc H$ be a subgroup of $\aut( X)$, and suppose that $\mc H$  is finitely generated relative to a finite collection of $0$-cell stabilizers.
Then there exists a connected $\mc H$-cocompact subcomplex $Y \hookrightarrow X$ with no missing $3$-shells.
\end{lem}
\begin{proof}
By Lemma~\ref{lem:Connected-Complex}, there exists a connected $\mc H$-cocompact subcomplex $Y$. If $Y$ has a missing $3$-shell, then, by hypothesis, one can replace $Y$ by another connected $\mc H$-cocompact subcomplex with strictly smaller perimeter. Since the perimeter is a non-negative integer, this process must terminate at a connected $\mc H$-cocompact subcomplex with no missing $3$-shells.
\end{proof}

\begin{proof}[Proof of Theorem~\ref{thm:lqc-criteria}]
By Lemma~\ref{lem:No-shells-Complex}, there is a connected $\mc H$-cocompact subcomplex $Y \subset X$ with no missing $3$-shells. By Proposition~\ref{lem:No-shells-convexity}, the inclusion $Y \rightarrow X$ is an $L$-quasi-isometric embedding.
\end{proof}

\section{Applications to high powered one-relator products}
\subsection{Background on one-relator products}
The natural framework for one-relator products is the relatively hyperbolic setting. We state Theorem~\ref{thm:filling} below to contextualize our most general result on one-relator products.  Theorem~\ref{thm:filling}.(\ref{filling-1}) is the ``Freiheitssatz for one-relator products'', and Theorem~\ref{thm:filling}.(\ref{filling-2}) is an immediate consequence of ``Newman Spelling Theorem''. We refer the reader to the survey article~\cite{DuHo94} by Duncan and Howie on one-relator products for an historical account of these ideas.  %In the modern language of relative hyperbolicity,  Theorem~\ref{thm:filling}.(\ref{filling-2})  is also a direct consequence of  Osin's  main result in~\cite{Os06-1} for $m$ sufficiently large.

\begin{thm}\label{thm:filling}
Let $\mc A$ and $\mc B$ be countable groups, and $r \in \mc A \ast \mc B$ a cyclically reduced word of length at least $2$.
If $m \geq 6$ then the following hold:
\begin{enumerate}
\item\label{filling-1}  The natural homomorphisms $\mc A \rightarrow \abrm$ and $\mc B \rightarrow \abrm$ are injective, and we regard
    $\mc A$ and $\mc B$ as subgroups.
    \item\label{filling-2}  The group $\abrm$ is  hyperbolic relative to  $\{\mc A, \mc B\}$.
\end{enumerate}
\end{thm}

\subsection{The Spelling Theorem and the Coned-off Cayley Complex $\widehat X$}
Let $\mc A$ and $\mc B$ be countable groups, let $r \in \mc A \ast \mc B$ a cyclically reduced word of length at least $2$ that is not a proper power, let $m>0$, and let $\mc G =\abrm$. The following was proven in \cite[Thm 3.1]{DuHo94-2}:
\begin{thm}[Spelling Theorem] \label{prop:filling}
Assume that $m \geq 6$. Let $w$ be a non-empty, cyclically reduced word belonging to the
normal closure of $r^m$. Then either:
\begin{enumerate}
\item $w$ is a cyclic permutation of $r^m$; or
\item $w$ has two strongly disjoint cyclic subwords $U_1$, $U_2$, such that each $U_i$ is
identical to a cyclic subword of $r^m$ of length at least $(m - 1)\ell - 1$.
\end{enumerate}
In particular, the length $|w|$ of the normal form of $w$ is at least $m\ell$.
\end{thm}

\begin{defn}[Coned-off Cayley Graph]
Let $\widehat \Gamma$ be the graph with vertex set equal  $\mc G \cup \{g\mc A : g\in \mc G\} \cup \{g\mc B : g\in \mc G\}$, i.e., there is a vertex for each element of $\mc G$, and a vertex for each left coset of $\mc A$ and $\mc B$. An element $g \in \mc G$ is connected to the left coset $f\mc A$ if and only if $g \in f\mc A$, and analogously $g$ is connected to $f\mc B$ if and only if $g \in f\mc B$. The resulting graph is called \emph{the coned-off Cayley graph of $\abrm$ with respect to $\{\mc A, \mc B\}$} (and with respect to the empty relative generating set). 

Observe that since $\mc A\cup \mc B$ is a generating set for $\abrm$, the graph $\widehat \Gamma$ is connected. Moreover each path in $\widehat \Gamma$ between elements of $\abrm$ is determined by its startpoint and an element of $\mc A \ast \mc B$.  
\end{defn}

\begin{defn}[Coned-off Cayley Complex]
We define the \emph{coned-off Cayley complex $\widehat X$ of $\abrm$} as follows:
The $1$-skeleton of $\widehat X$ is $\widehat \Gamma$. We add a single $2$-cell to $\widehat \Gamma$ for each closed cycle in $\widehat \Gamma$ labelled by $r^m$. We emphasize, that each such closed cycle corresponds to $m$~distinct closed paths, and so each $2$-cell has $\Z_m$ stabilizer under the $\abrm$ action.

Finally, we observe that when $|r| \geq 2$ and $m\geq6$, each $2$-cell in $\widehat X$ has embedded boundary cycle. Indeed, this follows from Theorem~\ref{prop:filling}.
\end{defn}

\begin{prop}\label{prop:free-product-complex}
If $|r| \geq 2$, $m\geq6$, then the Coned-off Cayley complex $\widehat X$ of $\abrm$ is simply-connected , is $m|r|$-circumscribed, is $|r|$-thin, and is a $C'(\frac{1}{m} + \epsilon)$-complex for each $\epsilon>0$.
\end{prop}
\begin{proof}
Let $Y_A,Y_B$ be standard 2-complexes of multiplication table presentations for $A,B$,
and let $Y=Y_A\vee Y_B$ denote their wedge, and let $Y_r$ be the space obtained by attaching an additional 2-cell
along $r^m$.
Let $\widetilde Y$ be its universal cover. Let $\hat Y$ denote the 3-complex obtained by coning-off each copy of
$\widetilde Y_A$ and $\widetilde Y_B$. We can collapse along free 2-faces and then along free 1-faces,
 so that only cone-edges remain. Note that the original 2-cell boundary cycles are homotoped to paths travelling in cone-edges.
Finally, each family consisting of $m$ two cells with common boundary is collapse to a single 2-cell. Observe that this does not affect simple connectivity, and we have constructed $\hat X$.

$\widehat X$ is $m|r|$-circumscribed since each $2$-cell has boundary cycle $r^m$.

Each $A$-syllable of $r$, corresponds to the concatenation of two $A$-cone-edges in $\widehat \Gamma$.
As there are  $\frac12|r|$ such $A$-syllables in $r$, we see that $\widehat X$ is $|r|$-thin.

The $C'(\frac{1}{m} + \epsilon)$ property is a variation of the well-known fact that if some word $u$ occurs twice in $r^m$,
then either these two occurrences are in the same $\Z_m$ orbit, or $|u|<|r|$.
\end{proof}

A notable difference with the standard case here is that if a syllable of $r$ has order 2, this leads to a length 2-piece in $\hat X$. Hence  $C'(\frac{1}{m})$ would not hold when  $|r|=2$. The reader is urged to consider the example $\Z_2\ast \Z_3$ and $(ab)^7$. The $\hat X$ is a subdivision of the $(7,3)$ tiling of the hyperbolic plane. In this case $\hat X$ is $C'(\frac{1}{7}+\epsilon)$ but not $C'(\frac{1}{7})$. 

\subsection{Proof of Theorem~\ref{thm:lq-products}}
\label{sub:lq-products}
\begin{thm}\label{thm:lq-products}
Let $\mc A$ and $\mc B$ be countable groups, and let $r \in \mc A \ast \mc B$ be a cyclically reduced word of length at least $2$.
Suppose that $3|r|<m$.

If $\mc H$ is a subgroup of $\abrm$ that is finitely generated relative to $\{\mc A, \mc B\}$,
then $\mc H$ is quasiconvex relative to $\{\mc A, \mc B\}$.
\end{thm}
\begin{proof}
By Proposition~\ref{prop:free-product-complex}, the Coned-off Cayley complex $\widehat X$ of $\mc G =\abrm$ with respect to $\{\mc A, \mc B\}$
is a $C'(\frac{1}{m} + \epsilon)$, simply-connected, uniformly circumscribed, and $|r|$-thin.  Since $ 3|r| <m$ and $\mc G$ acts without inversions on the 1-skeleton $\widehat \Gamma$ of $\widehat X$, the conclusion of Theorem~\ref{thm:local-quasiconvexity} for $\mc G$ and $\widehat \Gamma$ holds with $\lambda=\frac{1}{m}+\epsilon$ for some $\epsilon>0$.

The coned-off Cayley graph $\widehat \Gamma$ of $\mc G$, i.e., the one skeleton of $\widehat X$,  is a connected and fine hyperbolic graph
on which $\mc G$ acts cocompactly and with finite edge stabilizers, see for example~\cite[Prop. 4.2]{MaWi10b}.

Therefore, for each subgroup $\mc H <\mc G$ that is finitely generated relative to $\{\mc A, \mc B\}$,
there exists a connected and quasi-isometrically embedded $\mc H$-cocompact subcomplex of $\widehat \Gamma$.
By Theorem~\ref{def:ours}, such subgroups are quasiconvex relative to $\{\mc A, \mc B\}$.
\end{proof}

\bibliographystyle{plain}
\bibliography{xbib}

\begin{thebibliography}{10}

\bibitem{BO99}
B.~H. Bowditch.
\newblock Relatively hyperbolic groups.
\newblock {\em Internat. J. Algebra Comput.}, 22(3):1250016, 66, 2012.

\bibitem{DuHo94-2}
A.~J. Duncan and James Howie.
\newblock Spelling theorems and {C}ohen-{L}yndon theorems for one-relator
  products.
\newblock {\em J. Pure Appl. Algebra}, 92(2):123--136, 1994.

\bibitem{DuHo94}
Andrew~J. Duncan and James Howie.
\newblock One relator products with high-powered relators.
\newblock In {\em Geometric group theory, {V}ol.\ 1 ({S}ussex, 1991)}, volume
  181 of {\em London Math. Soc. Lecture Note Ser.}, pages 48--74. Cambridge
  Univ. Press, Cambridge, 1993.

\bibitem{Haglund1991}
Fr{\'e}d{\'e}ric Haglund.
\newblock Les poly\`edres de {G}romov.
\newblock {\em C. R. Acad. Sci. Paris S\'er. I Math.}, 313(9):603--606, 1991.

\bibitem{HruskaWise-Torsion}
G.~Christopher Hruska and Daniel~T. Wise.
\newblock Towers, ladders and the {B}. {B}. {N}ewman spelling theorem.
\newblock {\em J. Aust. Math. Soc.}, 71(1):53--69, 2001.

\bibitem{Ka33}
Egbert R.~Van Kampen.
\newblock On {S}ome {L}emmas in the {T}heory of {G}roups.
\newblock {\em Amer. J. Math.}, 55(1-4):268--273, 1933.

\bibitem{Ma08}
Eduardo Mart\'{\i}nez-Pedroza.
\newblock On quasiconvexity and relatively hyperbolic structures on groups.
\newblock {\em Geom. Dedicata}, 157:269--290, 2012.

\bibitem{MaWi10b}
Eduardo Mart{\'{\i}}nez-Pedroza and Daniel~T. Wise.
\newblock Relative quasiconvexity using fine hyperbolic graphs.
\newblock {\em Algebr. Geom. Topol.}, 11(1):477--501, 2011.

\bibitem{McWi-coherence}
J.~P. McCammond and Daniel~T. Wise.
\newblock Coherence, local quasiconvexity, and the perimeter of 2-complexes.
\newblock {\em Geom. Funct. Anal.}, 15(4):859--927, 2005.

\bibitem{McWi-fans}
Jonathan~P. McCammond and Daniel~T. Wise.
\newblock Fans and ladders in small cancellation theory.
\newblock {\em Proc. London Math. Soc. (3)}, 84(3):599--644, 2002.

\bibitem{WisePolygons}
Daniel~T. Wise.
\newblock The residual finiteness of negatively curved polygons of finite
  groups.
\newblock {\em Invent. Math.}, 149(3):579--617, 2002.

\end{thebibliography}

\end{document}